\documentclass[a4paper,reqno]{amsart}
\usepackage{amsmath,amsthm,amssymb,amsfonts,amsbsy}
\usepackage{enumitem,color,graphicx}
\usepackage[all,pdf]{xy}

\usepackage[backref=page,hyperindex=true,CJKbookmarks=true,
colorlinks,linkcolor=blue,anchorcolor=red,citecolor=cyan]{hyperref}

\usepackage{geometry}
\newtheorem*{theorem*}{Theorem}
\newtheorem*{proof*}{Proof}
\newtheorem*{proposition*}{Proposition}
\newtheorem*{notation*}{Notation}
\newtheorem*{corollary*}{Corollary}
\newtheorem*{claim*}{Claim}
\newtheorem*{remark*}{Remark}

\newtheorem{theorem}{Theorem}[section]
\newtheorem{proposition}{Proposition}[section]

\newtheorem{lemma}[proposition]{Lemma}
\newtheorem{claim}[proposition]{Claim}

\theoremstyle{definition}
\newtheorem{definition}[proposition]{Definition}
\theoremstyle{remark}
\newtheorem{remark}[proposition]{Remark}

\numberwithin{equation}{section}
\def\e{{\varepsilon}}

\def\NN{\mathbb{N}}
\def\RR{\mathbb{R}}
\def\ZZ{\mathbb{Z}}
\def\TT{\mathbb{T}}
\def\SS{\mathbb{S}}

\address[S. Gan]{School of Mathematical Sciences, Peking University, Beijing 100871, China}
\email{gansb@pku.edu.cn}

\address[Y. Shi]{School of Mathematical Sciences, Peking University, Beijing 100871, China}
\email{shiyi@math.pku.edu.cn}

\address[M. Xia]{School of Mathematical Sciences, Peking University, Beijing 100871, China}
\email{xiamy@pku.edu.cn}

\title[Density of Birkhoff sums]{On the density of Birkhoff sums for Anosov diffeomorphisms}
\author[S. Gan, Y. Shi and M. Xia]{Shaobo Gan, Yi Shi and Mingyang Xia}


\begin{document}
		
	\begin{abstract}
		Let $f:M \rightarrow M$ be an Anosov diffeomorphism on a nilmanifold. 
		We consider Birkhoff sums for a H\"older continuous observation along periodic orbits. 
		We show that if there are two Birkhoff sums distributed at both sides of zero, 
		then the set of Birkhoff sums of all periodic points is dense in $\RR$.
	\end{abstract}

\maketitle

\section{Introduction}

Dynamics is aimed to describe the long term evolution of systems under the known ``infinitesimal'' evolution rule.
The hyperbolicity plays one of the most important roles in the fields of differentiable dynamical systems.
When the dynamical system is driven by a strong hyperbolic diffeomorphism,
it exhibits chaotic behaviors which are the deep nature for many amazing phenomenons.

In this paper, we focus on a uniformly hyperbolic diffeomorphism $f:M \rightarrow M$, namely Anosov system (see Definition \ref{def:Anosov}).
Sigmund \cite[Theorem 1]{Si70} showed that, if $f$ is a transitive Anosov diffeomorphism, then all the Dirac measures of periodic orbits are dense in the space of $f$-invariant measures $\mathcal{M}_f$, i.e.,
$$
\left\lbrace\frac{1}{\pi(z)}\sum_{i=0}^{\pi(z)-1}\delta_{f^i(z)}~|~z \in \mathrm{Per}(f)\right\rbrace
$$
forms a dense subset of $\mathcal{M}_f$ in the sense of the weak-$*$ topology, where $\mathrm{Per}(f)$ is the set of periodic points of $f$. This implies for every continuous observation $\phi:M\rightarrow \RR$ and every measure $\mu\in\mathcal{M}_f$, there exists a sequence consisting of periodic points $z_n\in\mathrm{Per}(f)$, such that the \emph{Birkhoff average} of $\phi$ along the orbit of $z_n$ converges to the space average of $\phi$ on $\mu$:
$$
\frac{1}{\pi(z_n)}S_{\phi}f(z_n)=
\frac{1}{\pi(z_n)}\sum_{i=0}^{\pi(z_n)-1}{\phi(f^i(z_n))}
~\longrightarrow~
\int\phi~{\rm d}\mu,
\qquad {\rm as}~n\rightarrow\infty.
$$
Throughout this paper, we denote by
$$S_\phi f(z)= \sum_{i=0}^{\pi(z)-1}{\phi(f^i(z))}$$
the \emph{Birkhoff sum} of an observation $\phi$ along the orbit of $z$, for $z\in \mathrm{Per}(f)$ with the period $\pi(z)$.

The density of Dirac measures of periodic orbits is a crucial property in studying dynamical systems. Katok \cite{Ka80} showed that given a $C^{1+\alpha}$ diffeomorphism $f:M\rightarrow M$ and a hyperbolic ergodic measure $\mu\in\mathcal{M}_f$, there exists a sequence of hyperbolic periodic orbits contained in the support of $\mu$, such that the Dirac measures of these periodic orbits approximate $\mu$. This result has been generalized to $C^1$-diffeomorphisms when the Oseledets splitting is dominated \cite{ABC}.

The approximation of periodic measures implies the Birkhoff averages along periodic orbits converge to the integral along hyperbolic ergodic measure for any continuous function.  Later, Sun and Wang \cite{WS} enhanced Katok's work showing that the Lyapunov exponents of periodic orbits can also approximate the Lyapunov exponents of hyperbolic ergodic measure for $C^{1+\alpha}$-diffeomorphisms. See also \cite{ZS} for $C^1$-diffeomorphisms when the Oseledets splitting is dominated.

On the other hand, it is natural to ask: what are the distributions of all Birkhoff sums along periodic orbits? For transitive Anosov diffeomorphisms, the classical Liv\v{s}ic theorem \cite{Liv72} shows that if a H\"older continuous observation is not cohomologous to zero, then its Birkhoff sum along some periodic orbit is nonzero. Then the Anosov shadowing lemma \cite{An67} implies that there exists a sequence of periodic orbits whose Birkhoff sums of this observation tend to infinity. However, the distribution of all Birkhoff sums along periodic orbits for an Anosov diffeomorphism is totally unknown.

In this paper, we get a dense distribution of all Birkhoff sums along periodic orbits for a H\"older continuous observation.

\begin{theorem}\label{main}
	Let $f$ be an Anosov diffeomorphism on a nilmanifold $M=G/\Gamma$.
	Assume the observation $\phi:M \rightarrow \RR$ is H\"older continuous, and there are two points $p,q\in \mathrm{Per}(f)$ satisfying
	$$S_\phi f(p)<0<S_\phi f(q).$$
	Then the set $\big\{S_\phi f(z)~|~z\in \mathrm{Per}(f)\big\}$ is dense in $\RR$.
\end{theorem}

\begin{remark}
	In particular, this theorem holds for Anosov diffeomorphisms on the torus $\TT^n$.
	In \cite{GS19}, the same property of Birkhoff sums has been proved for expanding maps on the circle $\SS^1=\RR^1/\ZZ^1$ by using direct calculations for a $C^1$ observation, which is different from here.
\end{remark}

This paper will be organized as follows: In Section \ref{sec:Pre-back}, we recall some knowledge about hyperbolic diffeomorphisms, and prepare a simple lemma with an algebraic notion introduced; In Section \ref{sec:Main proof}, we show the full details of the proof of our main result.

\section{Preliminaries}\label{sec:Pre-back}

\subsection{Anosov diffeomorphisms}

In this subsection, we collect some properties of Anosov diffeomorphisms on nilmanifolds, which will be used in our proof later.

\begin{definition}\label{def:Anosov}
	Let $M$ be a compact smooth manifold without boundary. A diffeomorphism $f: M\to M$ is called \emph{Anosov}, if there exist a Riemannian metric $\|\cdot\|$ and a continuous $Df$-invariant splitting $TM=E^s \oplus E^u$ and a constant $\lambda\in(0,1)$, such that
	$$ \|Df(v^s)\|<\lambda<1<\lambda^{-1}<\|Df(v^u)\| $$
	for every unit vector $v^*\in E^*$, for $*=s,u$.
\end{definition}

\begin{remark}
	Equivalently, this definition indeed means the entire manifold $M$ is a \emph{hyperbolic set}.
\end{remark}

The first property we need is the shadowing property for Anosov diffeomorphisms. Recall that a sequence $\{x_n\}_{n\in\ZZ}$ of points in $M$ is a $\delta$-pseudo-orbit, if for any $n\in\ZZ$, $d(f(x_n),x_{n+1})\le\delta$; if there exists $m\in\NN$ such that for any $n\in\ZZ$, $x_{n+m}=x_n$, then it is called a periodic pseudo-orbit. A point $z\in M$ $\varepsilon$-shadows the pseudo-orbit $\{x_n\}_{n\in\ZZ}$ if for any $n\in\ZZ$, $d(f^n(z), x_n)\le \varepsilon$.

\begin{proposition}\cite[Theorem 1.2.3]{Pi99}\label{Hyper-prop}
	Let $f:M\rightarrow M$ be an Anosov diffeomorphism, then the system has Lipschitz pseudo-orbit shadowing property, i.e., there exist $\mu\ge 1$ and $\delta_0 > 0$, such that for any $\delta \in (0,\delta_0]$,  every  $\delta$-pseudo-orbit is $\mu\delta$-shadowed by some point. Furthermore, if the pseudo-orbit is periodic, then the shadowing point is also periodic.
\end{proposition}

A homeomorphism $f:M\to M$ is called \emph{transitive} if there exists a point whose positive semi-orbit under $f$ is dense in the entire space $M$.
For transitive Anosov diffeomorphisms, there is a classical result called Liv\v{s}ic Theorem (cf. \cite{Liv72}).
A proof of the simple version of Liv\v{s}ic Theorem below can be found in \cite{KH95} (see Theorem 19.2.1 including Remark) or \cite{KN11} (see Theorem 5.3.1)\footnote{In Theorem 5.3.1 of \cite{KN11}, since $\mathbb{S}^1$ is an Abelian group, the translation-invariant metric $d_H$ satisfying (5.2.1) exists automatically. It is the natural distance on $\mathbb{S}^1$. Moreover, since $\mathbb{S}^1$ is compact, the $\lambda$-center bunching condition in Theorem 5.3.1 is also satisfied. The discussion can be seen in the same page of Theorem 5.3.1 in \cite{KN11}.}.

\begin{theorem}\label{THM:Liv}[Liv\v{s}ic]
	Let $f$ be a transitive Anosov diffeomorphism on a compact smooth manifold $M$. Assume $\phi:M \rightarrow \RR$ or $\mathbb{S}^1$ is $\theta$-H\"older continuous, and $S_\phi f(x) =0$ for any $x\in \mathrm{Per}(f)$. Then there exists a $\theta$-H\"older continuous function $\psi$ with $\phi=\psi\circ f-\psi$. Moreover, $\psi$ is unique up to an additive constant.
\end{theorem}

Anosov diffeomorphisms were systematically studied since 1960s. For instance, see \cite{An67}, \cite{Sm67}, \cite{Fr70} and \cite{Man73} in some detail. One of the most typical examples is the Anosov toral automorphism that is also called \emph{Thom toral automorphism}. Precisely, any hyperbolic linear map $A:\RR^n\rightarrow\RR^n$ with $A\in {\rm GL}(n,\ZZ)$ will induce a quotient map on the $n$-torus $\TT^n=\RR^n/\ZZ^n$, and the induced map is exactly an Anosov toral automorphism. More general Anosov diffeomorphisms are hyperbolic nilmanifold automorphisms.

Let $G$ be a simply connected nilpotent Lie group and $\Gamma$ a uniform lattice of $G$, i.e., $\Gamma$ is a discrete subgroup of $G$ and $G/\Gamma$ is compact, then $M\triangleq G/\Gamma$ is called a \emph{nilmanifold}. Assume that $\phi: G\to G$ is a continuous automorphism with $\phi(\Gamma)=\Gamma$, $\phi$ will naturally induce a diffeomorphism $f: M\to M$. Denote by $\mathfrak{g}=T_eG$ the Lie algebra of $G$ and $\exp:\mathfrak{g}\to G$ the exponential map, where $T_eG$ is the tangent space of $G$ at the identity element $e \in G$. Let $\Phi=D_e\phi:\mathfrak{g}\to\mathfrak{g}$ be the Lie algebraic automorphism with $\phi\circ\exp=\exp\circ\Phi$. If $\Phi$ is hyperbolic, i.e., all its eigenvalues are not equal to 1 in modulus, then $f$ is called a \emph{hyperbolic nilmanifold automorphism}. Such hyperbolic nilmanifold automorphisms are standard examples of Anosov diffeomorphisms.

\begin{theorem}\cite[Theorem C]{Man74}\label{THM:A on N}
	If $f$ is an Anosov diffeomorphism on a nilmanifold $M=G/\Gamma$, then it is topologically conjugated to a hyperbolic nilmanifold automorphism. In particular, $f$ is transitive and $f_*:\pi_1(M)=\Gamma\circlearrowleft$ is hyperbolic, where $\pi_1(M)$ is the fundamental group of $M$.
\end{theorem}

\begin{remark}
	We are supposed to pay more attention to Theorem \ref{THM:A on N}.
	\begin{enumerate}\label{RMK:A}
		\item Let $f:M\to M$ be an Anosov diffeomorphism on a nilmanifold $M$. Since $f$ is transitive, it satisfies Theorem \ref{THM:Liv}. Moreover, the whole manifold $M$ is a homoclinic class of $f$. Namely, for any two periodic points $p$ and $q$, there are two points $x, y\in M$ such that
		\begin{align}
			x\in W^s(p)\pitchfork W^u(q) \quad {\rm and}& \quad y\in W^s(q)\pitchfork W^u(p).\nonumber
		\end{align}
		\item By \cite[Theorem 2.11]{Ra72}, the automorphism $f_*:\pi_1(M)=\Gamma\circlearrowleft$ can be uniquely extended as an automorphism (still denoted by) $f_*:G\circlearrowleft$. Theorem \ref{THM:A on N} claims that its induced diffeomorphism (still denoted by) $f_*:G/\Gamma\circlearrowleft$ is a hyperbolic nilmanifold automorphism, and the original Anosov diffeomorphism $f$ is topologically conjugate to $f_*$.
		\item Smale in \cite{Sm67} poses the problem of classifying all the Anosov diffeomorphisms on compact manifolds up to the topological conjugacy. The positive answer by Manning that extended the result of Franks in \cite{Fr69} is actually based on nilmanifolds. In fact, it should be noticed that the analogical conclusion may not be true for Anosov diffeomorphisms on \emph{infra-nilmanifolds}, which is a more general algebraic class of manifolds. And there is some discussion and progress made by Dekimpe and Hammerlindl (see \cite{De12} and \cite{Ham14} for commentary).
	\end{enumerate}
\end{remark}

\subsection{The asymptotically rational independence}

The final preparation is related to an algebraic notion. Recall that we say two real numbers $a$ and $b$ are \emph{rationally independent} if $k\cdot a+l\cdot b\not=0$ for any $k, l\in\ZZ$ with $|k|+|l|>0$. It is clear that if $a$ and $b$ are rationally independent, then the set $\big\{k\cdot a+l\cdot b~|~k,l\in\ZZ\big\}$ is dense near 0 (hence dense in $\RR$). Here is an asymptotic version of the rational independence introduced in \cite{GS19}.

\begin{definition}\label{def:ARI}
	Let $\{a_n\}_{n\in\NN}$ be a sequence of real numbers and $b$ a real number. The sequence $\{a_n\}_{n\in\NN}$ is called \emph{asymptotically rationally independent} of $b$, if there exist $0<\e_n\rightarrow0$ and $k_n,l_n\in\ZZ$, such that $$0<k_n\cdot a_n+l_n\cdot b<\e_n.$$
\end{definition}

The following properties about the asymptotic rational independence will be helpful later.
\begin{lemma}\label{lem:ARI}
	Let $\{a_n\}_{n\in\NN}$ be a sequence of real numbers, and $b\in\RR\setminus\{0\}$, such that $a_n/b=l_n/k_n$ with $l_n,k_n\in\ZZ$ and ${\it gcd}(l_n,k_n)=1$.
	\begin{enumerate}
		\item If $\{a_n\}_{n\in\NN}$ is asymptotically rationally independent of $b$, then $|k_n|\to\infty$ as $n\to\infty$. Moreover,
		$$
		\inf_{k,l\in\ZZ}\big\{ka_n+lb ~|~ ka_n+lb>0\big\}
		=\left|\frac{b}{k_n}\right|.
		$$
		\item If any subsequence of $\{a_n\}_{n\in\NN}$ is not asymptotically rationally independent of $b$, then there exists $c>0$ such that for every $n\in\NN$, $a_n=s_n\cdot c$ and $b=t\cdot c$ with $s_n,t \in\ZZ$.
	\end{enumerate}
\end{lemma}

\begin{proof}
	It is easily seen that $\{a_n\}$ is asymptotically rationally independent of $b$ if and only if
	$$
	\lim_{n\to\infty} \inf_{k,l\in\ZZ}\big\{ka_n+lb ~|~ ka_n+lb>0\big\}=0.
	$$
	Since $a_n/b=l_n/k_n$ with ${\it gcd}(l_n,k_n)=1$, we have
	$$
	\big\{ka_n+lb~|~k,l\in\ZZ\big\} = (b/k_n)\ZZ \triangleq
	\big\{ bm/k_n~|~m\in\ZZ \big\}.
	$$
	This directly leads to
	$$
	\inf_{k,l\in\ZZ}\big\{ka_n+lb ~|~ ka_n+lb>0\big\}
	=\left|\frac{b}{k_n}\right|
	\qquad {\rm and} \qquad
	\lim\limits_{n\to\infty} \left|k_n\right| = +\infty.
	$$
	
	If any subsequence of the sequence $\{a_n\}$ is not asymptotically rationally independent of $b$, there exists an upper bound $\mathrm{K}\in\NN$ for the related sequence $\big\{|k_n|\big\}_{n\in \NN}$. Direct calculations show that the positive number $c=|b/\mathrm{K}!|$ will satisfy the need of the desired conclusion. Namely, there are $s_n,t \in\ZZ$ such that $a_n=s_n\cdot c$ and $b=t\cdot c$. This ends the proof.
\end{proof}

\section{Proof of Theorem \ref{main}}\label{sec:Main proof}

\begin{proof}[Proof of Theorem \ref{main}]
	For any $K_0\in\RR$ and $\e >0$, we are aimed to show that there exists a periodic point $z\in \mathrm{Per}(f)$ such that $S_\phi f(z) \in (K_0-\e,K_0+\e)$.
	
	Since $f$ is a transitive Anosov diffeomorphism, by the item 1 of Remark \ref{RMK:A}, for these two given periodic points $p$ and $q$, we have two points
	\begin{align}
		x\in W^s(p)\pitchfork W^u(q) \quad {\rm and}& \quad y\in W^s(q)\pitchfork W^u(p).\nonumber
	\end{align}
	
	Note that $\phi:M \rightarrow \RR$ is a H\"older continuous function, i.e., there exist two constants $\theta\in(0,1)$ and $C>0$ such that for any $x,y\in M$,
	$$|\phi(x)-\phi(y)| \leq C\cdot d^\theta(x,y),$$
	where $d$ is the metric induced by the Riemannian structure on the manifold $M$.
	
	\begin{figure}[htbp]
		\centering
		\includegraphics[width=12.138cm]{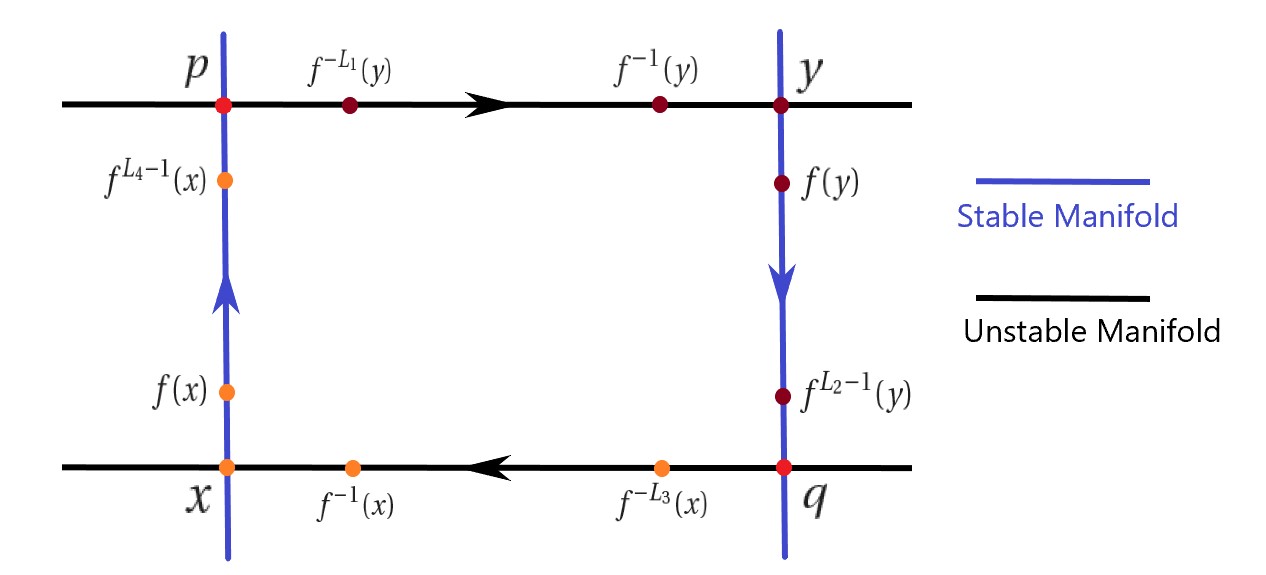}
		\caption{Periodic pseudo-orbit}
		\label{JZT}
	\end{figure}
	
	Take $L_i\in\NN, i=1, 2, 3, 4$, such that $L_1, L_4$ are multiples of $\pi(p)$ and $L_2, L_3$ are multiples of $\pi(q)$. Consider the periodic pseudo-orbit $Q$ as shown in Figure \ref{JZT},
	$$
	Q=\left\{f^{-L_1}(y), \cdots,f^{-1}(y),~y~,~f(y), \cdots, f^{L_2-1}(y),
	f^{-L_3}(x), \cdots,f^{-1}(x),~x~,~f(x), \cdots, f^{L_4-1}(x)\right\}.
	$$
	Denote by
	$$
	\delta_0=
	\max\left\{d(x,p),d(y,q),d(x,q),d(y,p)\right\},
	$$
	$$
	L_0=\min\big\{L_1,L_2,L_3,L_4\big\} \quad {\rm and} \quad L=\sum_{i=1}^4L_i,
	$$
	then the period of pseudo-orbit $Q$ is $L$.
	
	Since $x\in W^s(p)\pitchfork W^u(q)$ and $y\in W^s(q)\pitchfork W^u(p)$, there exists $H>0$, such that for every $n\geq0$,
	\begin{align*}
		d(f^n(x),f^n(p))\leq H\lambda^nd(x,p),\qquad
		& d(f^{-n}(x),f^{-n}(q))\leq H\lambda^nd(x,q); \\
		d(f^n(y),f^n(q))\leq H\lambda^nd(y,q),\qquad
		& d(f^{-n}(y),f^{-n}(p))\leq H\lambda^nd(y,p).
	\end{align*}
	This implies $Q$ is a $\delta$-pseudo-orbit, where
	$$\delta = 2 H \lambda^{L_0}\delta_0.$$
	
	In the following, we will always choose $L_i, i=1, 2, 3, 4$, such that there exists $\alpha\in(0,1)$ satisfying
	$$
	L_0\ge \alpha L_i,~i=1, 2, 3, 4,
	$$
	thus $\delta=2H\lambda^{L_0}\delta_0\leq 2H\lambda^{\alpha L_i}\delta_0$, for $i=1, 2, 3, 4.$
	
	By the Lipschitz pseudo-orbit shadowing property in Proposition \ref{Hyper-prop}, we are going to obtain a proper $L$-periodic point $z$ that $\mu\delta$-shadows the periodic $\delta$-pseudo-orbit $Q$. Without loss of generality, we denote by $z$ the point shadowing exactly the point $y\in W^s(q)\pitchfork W^u(p)$. Furthermore, in order to obtain the desired $L$-periodic point $z$ satisfying $S_\phi f(z) \in (K_0-\e,K_0+\e)$, we will take care of these iterations $L_i$'s with great patience later.
	
	In the light of Lemma \ref{lem:ARI}, we now divide our proof into the following lemmas by studying the relationships between $S_\phi f(p)$ and $S_\phi f(q)$ in the viewpoint of the rational independence.
	
	\begin{lemma}\label{Thm-1}
		Let $f:M\rightarrow M$ be a transitive Anosov diffeomorphism. Assume $\phi:M\rightarrow \RR$ is a H\"older continuous observation, and there exist $p,q\in\mathrm{Per}(f)$ with $S_\phi f(p)<0<S_\phi f(q)$. If $S_\phi f(p)$ and $S_\phi f(q)$ are rationally independent, then the set $\big\{S_\phi f(z)~|~z\in \mathrm{Per}(f)\big\}$ is dense in $\RR$.
	\end{lemma}
	
	\begin{proof}[Proof of Lemma \ref{Thm-1}]
		Fix $K_0\in\RR$ and $\e>0$, we want to show there exists a periodic point $z\in \mathrm{Per}(f)$ such that $S_\phi f(z)\in(K_0-\e,K_0+\e)$.	
		For simplicity, we consider $p$ and $q$ are fixed points at first.
		
		We start to deal with the H\"older continuous observation $\phi$ along the pseudo-orbit $Q$. Keeping the orbital segment $Q_1=\left\{f^{-L_1}(y),...,f^{-1}(y)\right\}$ with the choice of $ y\in W^s(q)\pitchfork W^u(p) $ in mind, we compute directly:
		$$
		\left|{\phi(f^{-i}(y))-\phi(f^{-i}(p))}\right|
		\leq C\cdot d^\theta(f^{-i}(y), f^{-i}(p))
		\leq { C\cdot (H\lambda^id(y,p))^\theta},
		$$
		which implies that the series $\sum_{i=1}^{+\infty}({\phi(f^{-i}(y))-\phi(f^{-i}(p))})$ is absolutely convergent. Denote by
		$$K_1 = \sum_{i=1}^{+\infty}
		\left({\phi(f^{-i}(y))-\phi(f^{-i}(p))}\right),$$
		for the convergence of the series, there exists $\tilde{L}_1\in\NN$ such that any $L_1\geq \tilde{L}_1$ implies
		\begin{align}\label{cal:pseudo-errors}
			\sum_{i=1}^{L_1}\left({\phi(f^{-i}(y))-\phi(f^{-i}(p))}\right)
			\in \left(K_1-\frac{\e}{9},K_1+\frac{\e}{9}\right).
		\end{align}
		
		Since the point $z$ $\mu\delta$-shadows this $\delta$-pseudo-orbit with $d(f^{-i}(z), f^{-i}(y))\le \mu\delta$, we have
		$$
		\sum_{i=1}^{L_1}\left|{\phi(f^{-i}(z))-\phi(f^{-i}(y))}\right|
		\leq  {L_1}\cdot C \cdot (\mu\cdot\delta)^\theta.
		$$
		Note that $\delta=2H\lambda^{L_0}\delta_0\leq 2H\lambda^{\alpha L_1}\delta_0 $ and $\lim\limits_{L_1\to +\infty} {L_1 \cdot \lambda ^{\theta \alpha L_1}}=0$. By enlarging $\tilde{L}_1$, for $L_1\geq \tilde{L}_1$, we have
		\begin{align}\label{cal:true-errors}
			\sum_{i=1}^{L_1}
			\left|{\phi(f^{-i}(z))-\phi(f^{-i}(y))}\right|
			\in \left(-\frac{\e}{9},\frac{\e}{9}\right).
		\end{align}
		
		Therefore, combining $(\ref{cal:pseudo-errors})$ and $(\ref{cal:true-errors})$, we get the estimation as follow:
		\begin{align}
			\sum_{i=1}^{L_1}\phi(f^{-i}(z))-L_1\cdot\phi(p) =\sum_{i=1}^{L_1}\left(\phi(f^{-i}(z))-\phi(f^{-i}(p))\right)\in \left(K_1-\frac{2\e}{9}, K_1+\frac{2\e}{9}\right).
		\end{align}
		Moreover, in the same style, we conclude that there exists $\tilde{L}_i$ large enough such that $L_i\geq \tilde{L}_i$ implies
		\begin{align}\label{cal:pt-errors}
			\sum_{i=1}^{L}\phi(f^{i}(z))-
			\left((L_1+L_4)\phi(p)+(L_2+L_3)\phi(q)\right)\in \left(K-\frac{8\e}{9},K+\frac{8\e}{9}\right),
		\end{align}
		where
		\begin{align}
			K_2 = \sum_{i=0}^{+\infty}\left({\phi(f^{i}(y))-\phi(f^{i}(q))}\right), \quad & \quad K_3 = \sum_{i=1}^{+\infty}\left({\phi(f^{-i}(x))-\phi(f^{-i}(q))}\right),\nonumber\\
			K_4 = \sum_{i=0}^{+\infty}\left({\phi(f^{i}(x))-\phi(f^{i}(p))}\right), \quad & \quad K=K_1+K_2+K_3+K_4.\nonumber
		\end{align}
		It should be noticed that this constant $K$ is closely related to the periodic points $p,q\in\mathrm{Per}(f)$.
		
		Since $\phi(p)$ and $\phi(q)$ are rationally independent, so are $2\phi(p)$ and $2\phi(q)$. For the constant $K_0-K$, there exist $m,n\in\NN$ large enough such that
		\begin{align}\label{app:RI-1}
			2m\phi(p)+2n\phi(q) \in
			\left(K_0-K-\frac{\e}{9},~ K_0-K+\frac{\e}{9}\right).
		\end{align}
		It is worth mentioning here that there are infinitely many $m,n\in\NN$ satisfying $(\ref{app:RI-1})$, i.e., for any $k\in \NN$, there exist $m_k, n_k\ge k$ such that
		\begin{align}\label{app:RI-2}
			2m_k\phi(p)+2n_k\phi(q) \in
			\left(K_0-K-\frac{\e}{9},~ K_0-K+\frac{\e}{9}\right).
		\end{align}
		
		Moreover, from the inequality (\ref{app:RI-2}) divided by $n_k$, we can get
		$$
		\dfrac{m_k}{n_k} \to -\dfrac{\phi(q)}{\phi(p)} \quad {\rm as} \quad k\to+\infty.
		$$
		Thus, for $k$ large enough, we take
		$$ L_1=L_4=m_k \quad {\rm and} \quad L_2=L_3=n_k, $$
		then we can get $\alpha\in(0,1)$ satisfying $L_0\ge \alpha L_i$ $(i=1,2,3,4)$, and $L_i\ge k$ $(i=1,2,3,4)$ satisfying all these above estimations contributed to the estimation $(\ref{cal:pt-errors})$ in advance.
		
		Hence, combining $(\ref{cal:pt-errors})$ and $(\ref{app:RI-2})$, we obtain the desired $L$-periodic point $z$ such that
		\begin{align}
			S_\phi f(z)=\sum_{i=1}^{L}\phi(f^{i}(z))\in (K_0-\e,K_0+\e).
		\end{align}
		
		Now we assume that $p$ and $q$ are general periodic points, and we just need deal with this situation under some slight modifications. Analogously, we can denote by $K=K_1+K_2+K_3+K_4$, where
		\begin{align}
			K_1 = \sum_{j=1}^{\pi(p)}\sum_{i=0}^{+\infty}
			\left({\phi(f^{-i\cdot\pi(p)-j}(y))- \phi(f^{-j}(p))}\right), \quad & \quad
			K_2 = \sum_{j=0}^{\pi(q)-1}\sum_{i=0}^{+\infty}
			\left({\phi(f^{i\cdot\pi(q)+j}(y))- \phi(f^{j}(q))}\right),\nonumber\\
			K_3 = \sum_{j=1}^{\pi(q)}\sum_{i=0}^{+\infty}
			\left({\phi(f^{-i\cdot\pi(q)-j}(x))- \phi(f^{-j}(q))}\right), \quad & \quad
			K_4 = \sum_{j=0}^{\pi(p)-1}\sum_{i=0}^{+\infty}
			\left({\phi(f^{i\cdot\pi(p)+j}(x))- \phi(f^{j}(p))}\right).\nonumber
		\end{align}
		
		Since $S_\phi f(p)$ and $S_\phi f(q)$ are rationally independent now, so are $2S_\phi f(p)$ and $2S_\phi f(q)$. For the constant $K_0-K$, for any $k\in\NN$, there exist $ m_k, n_k\ge k,$ such that
		\begin{align}\label{app:RI2}
			2m_k\cdot S_\phi f(p)+2n_k\cdot S_\phi f(q) \in \left(K_0-K-\frac{\e}{9},~ K_0-K+\frac{\e}{9}\right).
		\end{align}
		For $k$ large enough, let
		$$ L_1=L_4=m_k\cdot\pi(p) \quad {\rm and} \quad L_2=L_3=n_k\cdot\pi(q), $$
		then $L_i$'s satisfy all the following estimations in advance,
		\begin{align}
			\sum_{j=1}^{\pi(p)}\sum_{i=0}^{m_k-1}\left({\phi(f^{-i\cdot\pi(p)-j}(y))- \phi(f^{-j}(p))}\right) ~=~&\sum_{i=1}^{L_1}\phi(f^{-i}(y))-m_k\cdot S_\phi f(p) ~\in~ \left(K_1-\frac{\e}{9},K_1+\frac{\e}{9}\right),\nonumber\\
			\sum_{j=0}^{\pi(q)-1}\sum_{i=0}^{n_k-1}\left({\phi(f^{i\cdot\pi(q)+j}(y))- \phi(f^{j}(q))}\right) ~=~&\sum_{i=0}^{L_2-1}\phi(f^{i}(y))-n_k\cdot S_\phi f(q)  ~\in~ \left(K_2-\frac{\e}{9},K_2+\frac{\e}{9}\right),\nonumber\\
			\sum_{j=1}^{\pi(q)}\sum_{i=0}^{n_k-1}\left({\phi(f^{-i\cdot\pi(q)-j}(x))- \phi(f^{-j}(q))}\right) ~=~&\sum_{i=1}^{L_3}\phi(f^{-i}(x))-n_k\cdot S_\phi f(q)  ~\in~ \left(K_3-\frac{\e}{9},K_3+\frac{\e}{9}\right),\nonumber\\
			\sum_{j=0}^{\pi(p)-1}\sum_{i=0}^{m_k-1}\left({\phi(f^{i\cdot\pi(p)+j}(x))- \phi(f^{j}(p))}\right) ~=~&\sum_{i=0}^{L_4-1}\phi(f^{i}(x))-m_k\cdot S_\phi f(p)  ~\in~ \left(K_4-\frac{\e}{9},K_4+\frac{\e}{9}\right).\nonumber
		\end{align}
		Combining these estimations, we have
		\begin{align}\label{cal:pseudo-errors2}
			\sum_{i=-L_1}^{L_2-1}\phi(f^{i}(y))+\sum_{i=-L_3}^{L_4-1}\phi(f^{i}(x)) - \left( 2m_k\cdot S_\phi f(p)+2n_k\cdot S_\phi f(q) \right) ~\in~ \left(K-\frac{4\e}{9},K+\frac{4\e}{9}\right).
		\end{align}
		
		Then there exists an $L$-periodic point $z\in \mathrm{Per}(f)$ that $\mu\delta$-shadows the $\delta$-periodic pseudo-orbit $Q$, and the same calculations as contributed to $(\ref{cal:true-errors})$ show that
		\begin{align}\label{cal:true-errors2}
			\sum_{i=1}^{L}\phi(f^{i}(z))-\left(\sum_{i=-L_1}^{L_2-1}\phi(f^{i}(y))+ \sum_{i=-L_3}^{L_4-1}\phi(f^{i}(x))\right)\in \left(-\frac{4\e}{9},\frac{4\e}{9}\right).
		\end{align}
		
		Finally, combining $(\ref{app:RI2})$, $(\ref{cal:pseudo-errors2})$ and $(\ref{cal:true-errors2})$, we obtain the desired $L$-periodic point $z\in \mathrm{Per}(f)$ such that $S_\phi f(z)\in (K_0-\e, K_0+\e)$. Consequently, the set $\big\{S_\phi f(z) ~|~ z\in \mathrm{Per}(f)\big\}$ is dense in $\RR$. This ends the proof of Lemma \ref{Thm-1}.
	\end{proof}
	
	\begin{remark}
		We actually proved a more general result. Let $f:M\rightarrow M$ be a $C^1$ diffeomorphism, and $\phi:M\rightarrow\RR$ be a H\"older continuous observation. If there exist two hyperbolic periodic points $p,q$ with the same index satisfying
		\begin{itemize}
			\item $p$ and $q$ are homoclinic related,
			\item $S_\phi f(p)<0<S_\phi f(q)$,
			\item $S_\phi f(p)$ and $S_\phi f(q)$ are rationally independent,
		\end{itemize}
		then the set $\big\{S_\phi f(z) ~|~ z\in \mathrm{Per}(f)\big\}$ is dense in $\RR$.
		
		The idea for proving this claim is the following. Take $x\in W^s(p)\pitchfork W^u(q)$ and $y\in W^s(q)\pitchfork W^u(p)$, then the set
		$$
		\Gamma=
		{\rm Orb}(p)\cup{\rm Orb}(q)\cup{\rm Orb}(x)\cup{\rm Orb}(y)
		$$
		is a hyperbolic set. There exists $\eta>0$, such that the maximal invariant set $\Lambda$ that is contained in the $\eta$-neighborhood of $\Gamma$ is also hyperbolic. Since the Lipschitz pseudo-orbit shadowing property (Proposition \ref{Hyper-prop}) still works on $\Lambda$, we can apply the same argument as proving Lemma \ref{Thm-1}.
	\end{remark}
	
	From now on, we assume that $S_\phi f(p)$ and $S_\phi f(q)$ are rationally dependent. Given a sequence $\{p_n\}_{n\in\NN}$  of periodic points, if the sequence $\big\{S_\phi f(p_n)\big\}_{n\in\NN}$ is asymptotically rationally independent of $S_\phi f(q)$, then $\big\{S_\phi f(p_n)\big\}_{n\in\NN}$ is also asymptotically rationally independent of $S_\phi f(p)$. Note that if there exists $p_n\in\mathrm{Per}(f)$ which is rationally independent of $S_\phi f(q)$, the proof will be done by Lemma \ref{Thm-1}. Therefore, we just come to deal with the following situation.
	
	\begin{lemma}\label{Thm-2}
		Let $f:M\rightarrow M$ be a transitive Anosov diffeomorphism. Assume $\phi:M\rightarrow \RR$ is a H\"older continuous observation, and there exist $p,q\in\mathrm{Per}(f)$ with $S_\phi f(p)<0<S_\phi f(q)$. If $S_\phi f(p)$ and $S_\phi f(q)$ are rationally dependent, moreover, there exists a sequence $\{p_n\}_{n\in\NN}$ of periodic points such that
		\begin{itemize}
			\item the sequence $\big\{S_\phi f(p_n)\big\}_{n\in\NN}$ is asymptotically rationally independent of $S_\phi f(q)$,
			\item for any $n\in\NN$, $S_\phi f(p_n)$ is rationally dependent of $S_\phi f(q)$,
		\end{itemize}
		then the set $\big\{S_\phi f(z) ~|~ z\in \mathrm{Per}(f)\big\}$ is dense in $\RR$.
	\end{lemma}
	
	\begin{proof}[Proof of Lemma \ref{Thm-2}]
		Fix $K_0\in\RR$ and $\e>0$, we want to show there exists a periodic point $z\in \mathrm{Per}(f)$ such that $S_\phi f(z)\in(K_0-\e,K_0+\e)$. By the idea of proving Lemma \ref{Thm-1}, for the slightly different situation now, we can get the same estimation as (\ref{app:RI2}).
		
		From Definition \ref{def:ARI}, we may assume that the sequence $\big\{S_\phi f(p_n)\big\}_{n\in \NN}$ never contains 0. By taking subsequence if necessary, we assume $\big\{S_\phi f(p_n)\big\}_{n\in \NN}$ is a negative sequence with respect to the fact that $S_\phi f(q)>0$. Otherwise, we just need to consider that the positive sequence $\big\{S_\phi f(p_n)\big\}_{n\in \NN}$ is asymptotically rationally independent of $S_\phi f(p)$.
		
		Let $\frac{S_\phi f(p_n)}{S_\phi f(q)}=\frac{l_n}{k_n}$ with $-l_n,k_n>0$, and
		${\it gcd}(l_n,k_n)=1$. From the item 1 of Lemma $\ref{lem:ARI}$, we have
		$$
		\lim_{n\to\infty} \inf_{k,l\in\NN}
		\big\{ k\cdot S_\phi f(p_n)+l\cdot S_\phi f(q) ~|
		~ k\cdot S_\phi f(p_n)+l\cdot S_\phi f(q)>0\big\}
		=\lim_{n\to\infty}\frac{S_\phi f(q)}{k_n}=0.
		$$
		Therefore, for above $\e>0$, there exists $N>0$ such that $n\geq N$ implies
		$$
		\inf_{k,l\in\NN}
		\big\{ k\cdot S_\phi f(p_n)+l\cdot S_\phi f(q) ~|
		~ k\cdot S_\phi f(p_n)+l\cdot S_\phi f(q)>0\big\}
		=\frac{S_\phi f(q)}{k_n}<\frac{\e}{18}.
		$$
		
		For the choice of $n$ (hence $p_n$), we take
		\begin{align}
			x\in W^s(p_n)\pitchfork W^u(q) \quad {\rm and} \quad y\in W^s(q)\pitchfork W^u(p_n),\nonumber
		\end{align}
		and denote by $K_{p_n}=\tilde{K_1}+\tilde{K_2}+\tilde{K_3}+\tilde{K_4}$, where
		\begin{align}
			&\tilde{K_1} = \sum_{j=1}^{\pi(p_n)}\sum_{i=0}^{+\infty}
			\left({\phi(f^{-i\cdot\pi(p_n)-j}(y))- \phi(f^{-j}(p_n))}\right), \nonumber\\
			&\tilde{K_2} = \sum_{j=0}^{\pi(q)-1}\sum_{i=0}^{+\infty}
			\left({\phi(f^{i\cdot\pi(q)+j}(y))-\phi(f^{j}(q))}\right),\nonumber\\
			&\tilde{K_3} = \sum_{j=1}^{\pi(q)}\sum_{i=0}^{+\infty}
			\left({\phi(f^{-i\cdot\pi(q)-j}(x))- \phi(f^{-j}(q))}\right), \nonumber\\
			&\tilde{K_4} = \sum_{j=0}^{\pi(p_n)-1}\sum_{i=0}^{+\infty}
			\left({\phi(f^{i\cdot\pi(p_n)+j}(x))- \phi(f^{j}(p_n))}\right).\nonumber
		\end{align}
		Here $K_{p_n}$ is closely related to the above choice of $p_n$ (and $q,x,y$ naturally).
		
		Note that ${\it gcd}(l_n,k_n)=1$ with $-l_n,k_n>0$, there exist $m_0,n_0\in\NN$ large enough such that
		$$
		m_0\cdot l_n + n_0\cdot k_n=\left\{\begin{array}{ll} 1, & \mathrm{if}\ K_0-K_{p_n}\ge 0\\
			-1, &\mathrm{if}\ K_0-K_{p_n}<0\end{array},\right.
		$$
		This implies
		\begin{align}
			2m_0\cdot S_\phi f(p_n)+2n_0\cdot S_\phi f(q) = 2\frac{S_\phi f(q)}{k_n} (m_0\cdot l_n + n_0\cdot k_n) \in \left(-\frac{\e}{9},\frac{\e}{9}\right).
		\end{align}
		By Archimedean Property, there exists $h\in\NN$ ($h=1$ when $K_0-K_{p_n}$ coincidentally equals to $0$) such that
		\begin{align}\label{app:ARI-1}
			h\cdot\left[ 2m_0\cdot S_\phi f(p_n)+2n_0\cdot S_\phi f(q) \right] \in \left(K_0-K_{p_n}-\frac{\e}{9}, ~K_0-K_{p_n}+\frac{\e}{9} \right).
		\end{align}
		This corresponds to the estimation (\ref{app:RI-1}).
		
		Actually, for any $i_0\in\NN$, there exist $m_{i_0},n_{i_0}\ge i_0$ satisfying (\ref{app:ARI-1}). In fact, for $i_0\in\NN$, just take the sequence of integers
		$$
		m_{i_0}=m_0+k_n\cdot i_0 \qquad {\rm and} \qquad
		n_{i_0}=n_0-l_n\cdot i_0,
		$$
		we also have the same estimation:
		\begin{align}\label{app:ARI-2}
			h\cdot\left[ 2m_{i_0}\cdot S_\phi f(p_n)+2n_{i_0}\cdot S_\phi f(q) \right]
			\in \left(K_0-K_{p_n}-\frac{\e}{9}, ~K_0-K_{p_n}+\frac{\e}{9} \right).
		\end{align}
		This corresponds to the desired estimation (\ref{app:RI-2}), or  (\ref{app:RI2}).
		
		Moreover, by the inequality (\ref{app:ARI-2}) divided by $n_{i_0}$, we have
		$$
		\dfrac{m_{i_0}}{n_{i_0}}\to -\dfrac{S_\phi f(q)}{S_\phi f(p_n)} \quad {\rm as} \quad i_0\to+\infty.
		$$
		Thus for $i_0$ large enough, we take
		$$
		L_1=L_4=h\cdot m_{i_0}\cdot\pi(p_n) \quad {\rm and} \quad
		L_2=L_3=h\cdot n_{i_0}\cdot\pi(q),
		$$
		then we can also get $\alpha\in(0,1)$ satisfying $L_0\ge \alpha L_i$, for $i=1,2,3,4$.
		
		Furthermore, with the above choice of $x$ and $y$, we can consider the periodic pseudo-orbit:
		$$
		\tilde{Q}=\left\{f^{-L_1}(y), \cdots, y,\cdots, f^{L_2-1}(y),
		f^{-L_3}(x), \cdots,x,  \cdots, f^{L_4-1}(x)\right\}.
		$$
		And for large enough $i_0$, these $L_i$'s naturally satisfy in advance all those similar estimations which are contributed to the estimation $(\ref{cal:pt-errors})$, or (\ref{cal:pseudo-errors2}) plus (\ref{cal:true-errors2}).
		
		Hence combining these discussions, the same argument of Lemma \ref{Thm-1} shows that if we choose $i_0$ large enough and consider the shadowing periodic orbit $z$, then we will have $S_\phi f(z)\in(K_0-\e,K_0+\e)$. This ends the proof of Lemma \ref{Thm-2}.
	\end{proof}
	
	\begin{remark}
		The lemma also holds if $p,q$ and $\{p_n\}_{n\in\NN}$ are contained in a transitive hyperbolic set.
	\end{remark}
	
	\begin{lemma}\label{Thm-3}
		Let $f$ and $\phi$ be as in Theorem \ref{main}. If for any $z\in \mathrm{Per}(f)$, $S_\phi f(z)$ is rationally dependent of $S_\phi f(q)$, then there does exist a sequence $\{p_n\}_{n\in\NN}$ of periodic points such that the corresponding sequence $\big\{S_\phi f(p_n)\big\}_{n\in\NN}$ is asymptotically rationally independent of $S_\phi f(q)$.
	\end{lemma}
	
	\begin{proof}[Proof of Lemma \ref{Thm-3}]
		It will follow from the argument by contradiction.
		Supposing on the contrary, from the item 2 of Lemma $\ref{lem:ARI}$, we know that there exists $c>0$ such that
		$$
		\big\{S_\phi f(z)~|~z\in \mathrm{Per}(f)\big\} \subset c\ZZ=
		\big\{c\cdot m~|~m\in\ZZ\big\}.
		$$
		
		We denote by $\Phi:M \rightarrow \SS^1=\RR/{c\ZZ}$, the projection of $\phi$, which is defined as
		$$\Phi(z)=\phi(z)\mod c,$$
		then $\Phi$ has trivial observations along periodic orbits, namely, $S_\Phi f(z)=\bar{0}$ for any $z \in \mathrm{Per}(f)$.
		Moreover, according to Theorem \ref{THM:Liv}, we have a function $\Psi:M \rightarrow \SS^1$ such that
		\begin{align}\label{coh}
			\Phi=\Psi\circ f-\Psi.
		\end{align}
		
		We will lift the equation (\ref{coh}) from $\SS^1$ to $\mathbb{R}$, since we have the following claim.
		
		\begin{claim}\label{lift}
			The action $\Psi_*:\pi_1(M) \rightarrow \pi_1(\SS^1)$ induced by $\Psi:M\rightarrow \SS^1$ is trivial.
		\end{claim}
		
		\begin{proof}[Proof of Claim \ref{lift}]
			From the equation (\ref{coh}), by thinking the action on the fundamental group, we have
			\begin{equation}\label{Hom-equ}
				\Phi_*=\Psi_*\circ f_*-\Psi_*.
			\end{equation}
			It can be easily seen that $\Phi_*:\pi_1(M) \rightarrow \pi_1(\SS^1)$ is trivial. In fact, we know $\phi_*:\pi_1(M) \rightarrow \pi_1(\RR)$ and $\RR$ is contractible, i.e., $\pi_1(\RR)=0$, thus $\phi_*\equiv 0_*$ leads to the projection $\Phi_*\equiv 0_*$. Then we have
			\begin{equation}\label{Hom-equ2}
				\Psi_*\circ f_*=\Psi_*.
			\end{equation}
			
			Note that $f$ is Anosov on the nilmanifold $M=G/\Gamma$, the induced action $f_*:\pi_1(M)=\Gamma\circlearrowleft$ is hyperbolic by Theorem \ref{THM:A on N}. More precisely, the automorphism $f_*$ on $\Gamma$ can be uniquely extended to an automorphism on $G$, which is also denoted by $f_*$. Considering the induced map $D_ef_*:T_eG\rightarrow T_eG$, where $e$ is the identity element of the Lie group $G$, we know the linear map $D_ef_*$ has no eigenvalues with modulus 1.
			
			Since $\Psi_*:\Gamma \rightarrow \pi_1(\SS^1)=\ZZ\subseteq\RR$, by Theorem 2.11 of \cite{Ra72}, we know the homomorphism $\Psi_*$ can be uniquely extended as $\Psi_*:G \rightarrow\RR$. Therefore, the induced homology equation $(\ref{Hom-equ})$ also holds on the Lie group $G$. Moreover, from the equation $(\ref{Hom-equ2})$, we have
			$$ D_e\Psi_*\cdot D_ef_*=D_e\Psi_*.$$
			This implies $D_e\Psi_*\equiv0_*$ based on the hyperbolicity of $D_ef_*$, hence $\Psi_*\equiv 0_*$.
		\end{proof}
		
		Now we lift $\Psi:M \rightarrow \SS^1=\RR/{c\ZZ}$ to a real-valued function $\psi:M \rightarrow \RR$, and corresponding to the equation (\ref{coh}), we get $m\in\ZZ$ such that
		$$\phi=\psi\circ f-\psi+mc.$$
		Consequently, we have
		$$
		S_\phi f(p)=mc\pi(p),\qquad {\rm and} \qquad
		S_\phi f(q)=mc\pi(q),
		$$
		which contradicts the given condition $S_\phi f(p)<0<S_\phi f(q)$.
	\end{proof}
	
	All in all, these three lemmas complete the whole proof of Theorem \ref{main}.
\end{proof}


\section*{Acknowledgments}
We would like to thank Jinpeng An for the valuable discussion 
and pointing out the reference \cite{Ra72} about the theory of Lie group.
S. Gan is supported by NSFC 11771025 and 11831001.
Y. Shi is supported by NSFC 12071007 and 11831001.


\bibliography{Bib-XIA2020}
\bibliographystyle{amsalpha}

\end{document}